     \documentclass[a4paper, 12pt, reqno]{amsart}  \usepackage[top=34truemm,bottom=29truemm,left=30truemm,right=30truemm]{geometry}

\usepackage{amscd, amsmath, amssymb, amsthm, comment, color, bm}
\numberwithin{equation}{section}

\usepackage{url}

\theoremstyle{plain}
\newtheorem{theorem}{Theorem}[section]
\newtheorem{proposition}[theorem]{Proposition}
\newtheorem{corollary}[theorem]{Corollary}
\newtheorem{lemma}[theorem]{Lemma}

\theoremstyle{definition}
\newtheorem*{definition}{Definition}

\theoremstyle{remark}
\newtheorem{remark}[theorem]{Remark}
\newtheorem{question}[theorem]{Question}

\begin{document}

\title[The parallel transport map over homogeneous space]{The parallel transport map over \\ reductive homogeneous space}
\author[M. Morimoto]{Masahiro Morimoto}

 \address{Department of Mathematical Sciences, Tokyo Metropolitan University, 1-1 Minami-Osawa, Hachioji-shi, Tokyo 192-0397, JAPAN}

 \email{morimoto.mshr@gmail.com}

 \thanks{This work was partly supported by Grant-in-Aid for JSPS Research fellows (No.\ 23KJ1793), by MEXT Promotion of Distinctive Joint Research Center Program JPMXP0723833165 and by Osaka Metropolitan University Strategic Research Promotion Project (Development of International Research Hubs).}

 \makeatletter
 \@namedef{subjclassname@2020}{%
   \textup{2020} Mathematics Subject Classification}
 \makeatother


\maketitle


\begin{abstract}
We show that the parallel transport map over a reductive homogeneous space with natural torsion-free connection becomes an affine submersion with horizontal distribution. This generalizes one of the main results in the author's previous paper in the case of affine symmetric spaces. We also prove the compactness of the shape operators of the submanifold lifted by the parallel transport map. This improves a previous result by the author and generalizes some results of Terng-Thorbergsson and of Koike.  Furthermore we propose two definitions for the regularized mean curvatures of affine Fredholm submanifolds in Hilbertable spaces and discuss their relations to the parallel transport map. In particular, each fiber of the parallel transport map over a reductive homogeneous space is shown to be minimal in both senses.\end{abstract}





\section{Introduction}

In the 1980s, Palais and Terng \cite{PT88, Ter89} started to study submanifolds in Hilbert spaces. 
For an immersed submanifold $M$ of a (real, separable) Hilbert space $V$, the first and second fundamental forms, the shape operators and the normal connection are defined in the same way as in the finite dimensional case. In particular, for each normal vector $\xi$ at $p \in M$, the shape operator $A_\xi$ is a self-adjoint operator on the tangent space $T_p M$. 
Palais and Terng especially considered the case that $M$ has finite codimension in $V$ and showed that the following conditions are equivalent:
\begin{enumerate}
\item[(a)] The end point map $\eta: T^\perp M \rightarrow V$, $(p, \xi) \mapsto \varphi(p) + \xi$ is a Fredholm map, where $\varphi : M \rightarrow V$ denotes the immersion.
\item[(b)] The shape operator $A_\xi$ is a compact operator for any $(p, \xi) \in T^\perp M$.
\end{enumerate}
If these equivalent conditions are satisfied, then $M$ is called a \textit{Fredholm} submanifold of $V$. By the condition (b) we can use the spectral theory of compact self-adjoint operators to study the principal curvatures of Fredholm submanifolds. They also considered the \textit{proper} condition in view of Morse theory and investigated \textit{proper Fredholm} (PF) submanifolds in $V$.

In their study of PF submanifolds in Hilbert spaces, Palais and Terng \cite{PT88, Ter89} gave examples of PF submanifolds which are orbits of a Lie group action. To study those examples, they introduced a natural fibration $\Phi : V_\mathfrak{g} \rightarrow G$ defined over a connected compact Lie group $G$ with a bi-invariant Riemannian metric. Here $V_\mathfrak{g} := L^2([0,1], \mathfrak{g})$ denotes the Hilbert space of all $L^2$-maps from the interval $[0,1]$ to the Lie algebra $\mathfrak{g}$ of $G$. The map $\Phi$ is nowadays called the \textit{parallel transport map}. Later, Terng and Thorbergsson \cite{TT95} considered, for a Riemannian symmetric space $N = G/K$ of compact type with projection $\pi : G \rightarrow N$, the composition $\pi \circ \Phi : V_\mathfrak{g} \rightarrow G \rightarrow N$. We write $\Phi_{N}$ for $\pi \circ \Phi$ and call it the parallel transport map over $N$. They showed that $\Phi$ (and thus $\Phi_N$) is a Riemannian submersion. Moreover they showed that for a compact submanifold $M$ of $N$, the inverse image $\hat{M} := \Phi_{N}^{-1}(M)$ is a PF submanifold of $V_\mathfrak{g}$ \cite[Lemma 5.8]{TT95}. 
Furthermore they showed that under some assumptions, $M$ is equifocal if and only if $\hat{M}$ is isoparametric. Based on these results, they investigated equifocal submanifolds in $N$ via isoparametric PF submanifolds in the \textit{flat} space $V_\mathfrak{g}$.

Note that all the submanifolds considered above are \textit{Riemannian} submanifolds. On the other hand, in connection with the parallel transport map over a non-compact space, we will encounter pseudo-Riemannian submanifolds or more generally, (affine) submanifolds with transversal bundles. Here a transversal bundle of an immersed submanifold $M$ means a vector bundle over $M$ complementary to $TM$ (see \cite[Section 3]{M6} for details).
Thus we generalize the above conditions (a) and (b) to such cases: 
Let $V$ be a \textit{Hilbertable} space, that is, a (real, separable) topological vector space whose topology is induced by a complete inner product. Equip $V$ with the flat connection $D$. 
Let $\varphi : M \rightarrow V$ be an immersed submanifold (i.e.\ immersion) of finite codimension with transversal bundle $\mathcal{W}$. 
The conditions (a) and (b) are generalized as follows:
\begin{enumerate}
\item[(A)] The end point map $\eta : \mathcal{W} \rightarrow V$, $(p, \xi) \mapsto \varphi(p) + \xi$ is a Fredholm map of index $0$ 
(i.e. $\dim \operatorname{Ker} (d\eta)_{(p, \xi)} = \dim \operatorname{Coker} (d\eta)_{(p, \xi)} < \infty$ for any $(p, \xi) \in \mathcal{W}$).

\item[(B)] The shape operator $A_\xi$ is a compact operator for any $(p, \xi) \in \mathcal{W}$.
\end{enumerate}
It follows that (A) is equivalent to the condition that $\operatorname{id} - A_\xi$ is a Fredholm operator of index $0$ for any $(p, \xi) \in \mathcal{W}$ \cite[Remark 4.5]{M6}. Thus, in the Riemannian case, the property of index $0$ is automatically satisfied since $A_\xi$ is self-adjoint. Moreover it follows from the general theory of Fredholm operators that (B) implies (A). However it is not clear whether (A) implies (B) or not. Note also that there seems to be no natural definition for a non-Riemannian submanifold of a Hilbertable space to be proper.

After the work of Terng and Thorbergsson \cite{TT95}, Koike \cite{Koi04} studied the parallel transport map over a Riemannian symmetric space $N = G/K$ of non-compact type. He equipped $G$ with the induced bi-invariant \textit{pseudo-Riemannian} metric and $V_\mathfrak{g}$ with the induced \textit{indefinite} $L^2$-inner product. 
He defined the parallel transport map $\Phi_{N} : V_\mathfrak{g} \rightarrow G \rightarrow N$ by a similar way and showed that it is a pseudo-Riemannian submersion. Moreover, he showed that for a submanifold $M$ of $N$, the pseudo-Riemannian submanifold $\hat{M} := \Phi_{N}^{-1}(M)$ of $V_\mathfrak{g}$ (with normal bundle $T^\perp \hat{M}$) satisfies (B) \cite[Theorem 5.8]{Koi04}. Based on these results, he extended the result of Terng and Thorbergsson \cite{TT95} to the case of Riemannian symmetric spaces of non-compact type. 

Recently, the author \cite{M6} studied the parallel transport map over a symmetric space $N = G/K$ which is not necessarily Riemannian or semisimple. He equipped $N$ with the canonical connection $\nabla^N$ and the \textit{Hilbertable} space $V_\mathfrak{g}$ with the flat connection $D$. He defined the parallel transport map $\Phi_N : V_\mathfrak{g} \rightarrow G \rightarrow N$ by a similar way and showed that it is an affine submersion with horizontal distribution $\mathcal{H}$ in the sense of Abe and Hasegawa \cite{AH01} and  the fundamental tensor $\mathcal{A}$ restricted to $\mathcal{H} \times \mathcal{H}$ is alternating \cite[Theorem 1.1]{M6}. Based on these properties, he showed  that for an immersed submanifold $\varphi : M \rightarrow N$  with transversal bundle $\mathcal{W}$, the lifted submanifold $\hat{\varphi} : \hat{M} \rightarrow V_\mathfrak{g}$ with lifted transversal bundle $\hat{\mathcal{W}}$ satisfies (A) \cite[Proposition 1.2]{M6}. (Here $\hat{M} := \{(p, u) \in M \times V_\mathfrak{g} \mid \varphi(p) = \Phi_N(u)\}$, $\hat{\varphi} (p, u) := u$ and $\hat{\mathcal{W}}$ is the horizontal lift of $\mathcal{W}$. If $M \subset N$ is a submanifold, then $\hat{M}$ is identified with $\Phi_N^{-1}(M)$.)
However, he was not able to show (B). Moreover he was not able to investigate the parallel transport map $\Phi_N$ over a reductive homogeneous space $N$ because he only considered the canonical connection on $N$, which is not torsion-free when $N$ is a non-symmetric space. 

The main purpose of the present paper is to investigate the parallel transport map $\Phi_N$ over a reductive homogeneous space $N = G/K$ and to show the property (B) for the lifted submanifold $\hat{\varphi} : \hat{M} \rightarrow V_\mathfrak{g}$. To do this we consider another affine connection on $N$, namely the \textit{natural torsion-free connection} \cite{Nom54, KNII}. The canonical connection and the natural torsion-free connection share the same geodesics. Those connections coincide when $N$ is a symmetric space. We will prove the following theorem  on the affine geometric properties of $\Phi_N$, which generalizes the author's previous result \cite[Theorem 1.1]{M6} in the case of affine  symmetric spaces. 

\begin{theorem}\label{mainthm1}\normalsize
Let $N = G/K$ be a reductive homogeneous space with decomposition $\mathfrak{g} = \mathfrak{k} \oplus \mathfrak{p}$ and with natural torsion-free connection $\nabla^{N}$. Then the parallel transport map $\Phi_{N} : (V_\mathfrak{g}, D) \rightarrow (N, \nabla^{N})$ is an affine submersion with horizontal distribution $\mathcal{H}^{\Phi_N}$. Moreover the fundamental tensor $\mathcal{A}^{\Phi_N}$ restricted to $\mathcal{H}^{\Phi_N} \times \mathcal{H}^{\Phi_N}$ is alternating.
\end{theorem}
Here $\mathfrak{k}$ denotes the Lie algebra of $K$ and $\mathfrak{p}$ an $\operatorname{Ad}(K)$-invariant subspace of $\mathfrak{g}$. The decomposition $\mathfrak{g} = \mathfrak{k} \oplus \mathfrak{p}$ induces a horizontal distribution $\mathcal{H}^{\Phi_N}$ of $\Phi_{N}$, see \eqref{hdist98} for details. To prove Theorem \ref{mainthm1} we need an accurate understanding of the work of Nomizu \cite{Nom54}. We will give a brief survey of this in modern terminology (Section \ref{affineconn}).

We will prove the following theorem on the property (B) for the lifted submanifold. This improves the author's result \cite[Proposition 1.2]{M6} and generalizes the results of Terng-Thorbergsson \cite[Lemma 5.8]{TT95} when $N$ is compact and of Koike \cite[Theorem 5.8]{Koi04} when $N$ is a Riemannian symmetric space of non-compact type.
\begin{theorem}\label{mainthm2}\normalsize
Let $N = G/K$ be a reductive homogeneous space with decomposition $\mathfrak{g} = \mathfrak{k} \oplus \mathfrak{p}$ and $\varphi :M \rightarrow N$ be an immersed submanifold with transversal bundle $\mathcal{W}$. Then the lifted submanifold $\hat{\varphi} : \hat{M} \rightarrow V_\mathfrak{g}$ with lifted transversal bundle $\hat{\mathcal{W}}$ satisfies \textup{(B)} (and thus \textup{(A)}).
\end{theorem}
Note that the proofs of Terng-Thorbergsson \cite{TT95} and Koike \cite{Koi04} do not apply to our general case since their proofs rely on calculations of the eigenvalues of the shape operators of $\hat{M}$, which are based on the root space decomposition associated to the compact Lie group $G$ or the Riemannian symmetric space $N$. We give a different, simpler proof which does not rely on such calculations or decompositions.

Based on Theorems \ref{mainthm1} and \ref{mainthm2} we furthermore study the mean curvature (i.e.\ the trace of the shape operator) of an immersed submanifold of a Hilbertable space satisfying (B). Note that the shape operators are not of trace class in general. 
In the Riemannian case, there are three definitions for the mean curvatures, see King-Terng \cite{KT93}, Heintze-Liu-Olmos \cite{HLO06} and Koike \cite{Koi02}. In the affine case, we will propose two definitions:\ one generalizes that of Heintze-Liu-Olmos \cite{HLO06} and the other is valid only for submanifolds lifted by the parallel transport map. Based on the latter definition, we will prove Theorem \ref{mainthm3}, which can be thought of an extension of results in the Riemannian case (see Remark \ref{main3remark}).  We will also show that each fiber of the parallel transport map over a reductive homogeneous space is minimal (i.e.\ having zero mean curvatures) regarding both of our definitions (Proposition \ref{minimalfiber} and Corollary \ref{minimalfiber89}).

The paper is organized as follows: In Section \ref{affineconn} we give a brief survey of the work of Nomizu \cite{Nom54}. In Section \ref{ptm} we study the parallel transport map over a reductive homogeneous space and prove Theorem \ref{mainthm1}. In Section \ref{compactshape} we study  the shape operators of fibers of the parallel transport map and prove Theorem \ref{mainthm2}. 
In Sections \ref{minimality} and \ref{finalsec} we consider the mean curvatures of submanifolds in Hilbertable spaces satisfying (B) and study their relations to the parallel transport map.

\section{Invariant affine connections on reductive homogeneous spaces}\label{affineconn}

In this section, we give a brief survey on invariant affine connections on reductive homogeneous spaces. We essentially follow the paper of Nomizu \cite{Nom54} where a simple and direct formulation is given. We also refer to Kobayashi-Nomizu \cite{KNII} for a formulation based on the theory of connections on principal fiber bundles.

Let $G$ be a connected Lie group and $K$ a closed subgroup of $G$. The coset manifold $N := G/K$ is called a \textit{homogeneous space}. The projection $\pi : G \rightarrow N$ becomes a principal $K$-bundle. Write $\mathfrak{g}$ and $\mathfrak{k}$ for the Lie algebras of $G$ and $K$ respectively. $N$ is called \textit{reductive} if there exists an $\operatorname{Ad}(K)$-invariant subspace $\mathfrak{p}$ of $\mathfrak{g}$ satisfying $\mathfrak{g} = \mathfrak{k} \oplus \mathfrak{p}$. 
Note that $\mathfrak{p}$ is isomorphic to $T_{eK} N$ and the representation $\operatorname{Ad} : K \rightarrow GL(\mathfrak{p})$ is equivalent to the isotropy representation $K \rightarrow GL(T_{eK} N)$. 
When we speak of a reductive homogeneous space, we always fix such a decomposition $\mathfrak{g} = \mathfrak{k} \oplus \mathfrak{p}$. 
Write $X_\mathfrak{k}$ and $X_\mathfrak{p}$ for the $\mathfrak{k}$- and $\mathfrak{p}$-components of $X \in \mathfrak{g}$ respectively. 

Let $N = G/K$ be a reductive homogeneous space with decomposition $\mathfrak{g} = \mathfrak{k} \oplus \mathfrak{p}$. The following lemma is elementary and not explicitly stated in \cite{Nom54}. However this is important in understanding the formulation.
\begin{lemma}\label{lem1}\normalsize
There exists a local trivialization $\phi : \pi^{-1}(U) \rightarrow U \times K$ around $eK$ such that $\tilde{U} := \phi^{-1}(U \times \{e\})$ is equal to $\exp W$ for an open neighborhood $W$ of $0$ in $\mathfrak{p}$.
\end{lemma}
\begin{proof}\normalsize
By the inverse function theorem there exists an open neighborhood $W$ of $0 \in \mathfrak{p}$ such that $\pi:  \exp W \rightarrow \pi(\exp W) = : U$ is a diffeomorphism. This defines the local smooth section of $\pi : G \rightarrow N$ which defines the desired local trivialization.
\end{proof}

Fix the above local trivialization. Then $\pi : \tilde{U} \rightarrow U$ is a diffeomorphism. For each $X \in \mathfrak{p}$ the vector field $X^\#$ on $U$ is defined by $X^\#_{aK} := d \pi (X_{a})$ where $a \in \tilde{U}$. Then 
\begin{equation}\label{leftinv}
X^{\#}_{aK} = dL_a(X^\#_{eK})
\end{equation}
for $a \in \tilde{U}$. Here $L_a$ denotes the left translation on $N = G/K$ by $a$. The following fact is essentially shown in \cite[p.\ 42]{Nom54} (where $X^\#$ is written as $X^*$).

\begin{lemma}[Nomizu \cite{Nom54}]\label{lem2}\normalsize
For each $k \in K$ there exists a neighborhood $\tilde{U}_1$ of $e \in \tilde{U}$ satisfying $k \tilde{U}_1 k^{-1} \subset \tilde{U}$. Moreover, 
\begin{enumerate}
\item $L_k(U_1) \subset U$ where $U_1 := \pi(\tilde{U}_1)$, 
\item $dL_k (X^\#) = (\operatorname{Ad}(k)X)^\#$ holds on $L_k(U_1)$ for any $X \in \mathfrak{p}$.
\end{enumerate}
\end{lemma}
In fact, since $\operatorname{Ad}(k)^{-1}\mathfrak{p} \subset \mathfrak{p}$ we can take an open neighborhood $W_1$ of $0 \in \mathfrak{p}$ satisfying $0 \in W_1 \subset W \cap \operatorname{Ad}(k)^{-1}W$. Then $\tilde{U}_1 := \exp W_1$ is the desired neighborhood. Then (i) is obvious and (ii) can be proven by  use of \eqref{leftinv}.

For a vector bundle $E$ over a reductive homogeneous space $N = G/K$ we write $\Gamma(E)$ for the set of all smooth sections of $E$. $\Gamma(TN)$ is also written as $\mathfrak{X}(N)$. An affine connection $\nabla: \Gamma(TN) \rightarrow \Gamma(TN \otimes T^*N)$ on $N$ is called \textit{$G$-invariant} if 
\begin{equation*}
\nabla_{dL_a (v)} dL_a (Z) = dL_a (\nabla_{v} Z)
\end{equation*}
for $Z \in \Gamma(TN)$, $v\in TN$ and $a \in G$. A bilinear map $\alpha : \mathfrak{p} \times \mathfrak{p} \rightarrow \mathfrak{p}$ is called \textit{$\operatorname{Ad}(K)$-invariant} if
\begin{equation*}
\alpha(\operatorname{Ad}(k)X, \operatorname{Ad}(k)Y) = \operatorname{Ad}(k)\alpha(X,Y)
\end{equation*}
for $X, Y \in \mathfrak{p}$ and $k \in K$. Such an $\alpha$ is also called a \textit{connection function}.

\begin{proposition}[Nomizu \cite{Nom54}]\normalsize
Let $N = G/K$ be a reductive homogeneous space with decomposition $\mathfrak{g} = \mathfrak{k} \oplus \mathfrak{p}$. Then there exists a one-to-one correspondence between the set of all $G$-invariant affine connections on $N$ and the set of all $\operatorname{Ad}(K)$-invariant bilinear maps $\mathfrak{p} \times \mathfrak{p} \rightarrow \mathfrak{p}$. The correspondence is given by 
\begin{equation}\label{relation}
\nabla_{X^\#_{eK} } Y^\#= \alpha(X, Y).
\end{equation}
\end{proposition}

In fact, by taking a basis $Y_1, \cdots, Y_n$ of $\mathfrak{p}$ we can express each $Z \in \mathfrak{X}(U)$ as $Z  = \sum_{i = 1}^n \varphi_i Y_i^\#$ where $\varphi_i \in C^\infty(U)$. Then the Leibniz rule shows
\begin{equation}\label{localexp}
\nabla_{X^\#_{eK}} Z
= 
\sum_{i =1}^n (\varphi_i(eK) \nabla_{X^\#_{eK}} Y_i^\# + (X^\#_{eK} \varphi_i) (Y_i^\#)_{eK}), 
\end{equation}
which implies that $\nabla$ is uniquely determined by $\alpha$. The $G$-invariance of $\nabla$ is equivalent to the $\operatorname{Ad}(K)$-invariance of $\alpha$, which can be seen by using Lemma \ref{lem2} (ii).

For $X \in \mathfrak{p}$ we write $X^*$ for the fundamental vector field on $N = G/K$, namely
\begin{equation*}
X^*_{aK} := \left. \frac{d}{dt} \right|_{t = 0} (\exp tX)aK.
\end{equation*}
Since $X^*_{eK} = X^\#_{eK}$ and $[X^*, Y_i^\#]_{eK} = 0$ it follows from \eqref{localexp} that: 
\begin{corollary}\label{canocon9}\normalsize
$
\nabla_{X^*_{eK}} Z
=
[X^*, Z]_{eK}
+
\alpha(X, Z_{eK})
$
for $X \in \mathfrak{p}$ and $Z \in \mathfrak{X}(N)$.
\end{corollary}

Let $\nabla$ be an affine connection on a reductive homogeneous space $N = G/K$ with connection function $\alpha$. Then: 
\begin{enumerate}
\item $\nabla$ is called the \textit{canonical connection} (or the canonical connection of the second kind) if 
$
\alpha(X, Y) = 0
$
for any $X, Y \in \mathfrak{p}$. 

\item $\nabla$ is called the \textit{natural torsion-free connection} (or the canonical connection of the first kind) if 
$
\alpha(X, Y) = \frac{1}{2} [X, Y]_\mathfrak{p}
$
for $X, Y \in \mathfrak{p}$.
\end{enumerate}
In both cases, $\gamma(t) := \pi (\exp tX)$ is a geodesic through $eK$ where $X \in \mathfrak{p}$. In the case (i), the parallel translation of $v \in T_{eK} N$ along $\gamma$ is equal to $dL_{\gamma(t)}(v)$. In the case (ii), the torsion tensor vanishes. When $N$ is a symmetric space, (i) and (ii) are equivalent.

 Let $G$ be a connected Lie group. Write $\Delta G$ for the diagonal of $G \times G$. Then $\rho : (G \times G) / \Delta G \rightarrow G$, $(a,b) \mapsto ab^{-1}$ is an isomorphism between symmetric spaces. 
\begin{corollary}\label{canoconng} \normalsize
The canonical connection $\nabla^G$ of $G \cong (G \times G) / \Delta G$ 
is given by
$
(\nabla^G_X Z)_e = \frac{1}{2} [X + X^R, Z]_e
$
where $X \in \mathfrak{g}$ and $Z \in \mathfrak{X}(G)$. 
Here $X^R$ is the right invariant vector field satisfying $X^R_e = X_e$. 
In particular $\nabla^G_X Y = \frac{1}{2} [X, Y]$ for $X, Y \in \mathfrak{g}$. 
\end{corollary}

\begin{proof}\normalsize
The decomposition associated to $(G \times G)/ \Delta G$ is given by $\mathfrak{g} \oplus \mathfrak{g} = \Delta \mathfrak{g} \oplus (\Delta \mathfrak{g})^\perp$ where $(\Delta \mathfrak{g})^\perp := \{(X, - X) \mid X \in \mathfrak{g}\}$. With respect to the $(G \times G)$-action on $G$ defined by $(b,c) \cdot a := bac^{-1}$ we have
$
(X, -X)^*_a = \left. \frac{d}{dt} \right|_{t = 0} \exp t(X, -X) \cdot a = X_a + X^R_a
$.
In particular $(X, -X)^*_e = 2 X_e$. 
Since $\rho$ is equivariant with respect to the $(G \times G)$-actions, we have 
$
2 (\nabla_X Z)_e
=
(\nabla_{(X, -X)^*} Z)_e
\cong
(\nabla_{(X, -X)^*} Z)_{(e,e)}
=
[(X, -X)^*, Z]_{(e,e)}
\cong 
[X + X^R, Z]_e
$
by Corollary \ref{canocon9}. 
\end{proof}

From Corollary \ref{canoconng} we see that the canonical connection of $G \cong (G \times G)/ \Delta G$ is the same as the natural torsion-free connection of $G \cong G / \{e\}$.

\section{The parallel transport map}\label{ptm}

In this section we study the parallel transport map over a reductive homogeneous space and prove Theorem \ref{mainthm1}. We refer to \cite[Section 6]{M6} and its references for the parallel transport map over a Lie group.

Let $G$ be a connected Lie group with Lie algebra $\mathfrak{g}$. Write $\mathcal{G} := H^1([0,1], G)$ for the Hilbert Lie group of Sobolev $H^1$-maps from $[0,1]$ to $G$ and $V_\mathfrak{g} := L^2([0,1], \mathfrak{g})$ for the Hilbertable space of all $L^2$-maps from $[0,1]$ to $\mathfrak{g}$. The adjoint representation $\operatorname{Ad}: G \rightarrow GL(\mathfrak{g})$ induces the representation $\mathcal{G} \rightarrow GL(V_\mathfrak{g})$ which is still written as $\operatorname{Ad}$. 
Write $l_g$ and $r_g$ for the left and right translations by $g \in \mathcal{G}$ respectively. Let $g'$ denote the weak derivative of $g$. 
The affine action of $\mathcal{G}$ on $V_\mathfrak{g}$ is defined by 
\begin{equation}\label{gauge}
g * u := \operatorname{Ad}(g) u - dr_g^{-1} (g')
\end{equation}
where $g \in \mathcal{G}$ and $u \in V_\mathfrak{g}$.
This action is transitive.  For a submanifold $U$ of $G \times G$, 
\begin{equation*}
P(G, U) := \{g \in \mathcal{G} \mid (g(0), g(1)) \in U\}
\end{equation*}
is a submanifold of $\mathcal{G}$. If $U$ is a Lie subgroup of $G \times G$, then $P(G, U)$ is a Lie  subgroup of $\mathcal{G}$ and acts on $V_\mathfrak{g}$ by \eqref{gauge}.  If $U = \{e\} \times G$ or $G \times \{e\}$, then the $P(G, U)$-action on $V_\mathfrak{g}$ is simply transitive,  where $e$ denotes the identity element of $G$.

The \textit{parallel transport map $\Phi : V_\mathfrak{g} \rightarrow G$ over $G$} is a submersion defined by $\Phi(u) := g_u(1)$ where $g_u \in \mathcal{G}$ is the unique solution to the ordinary differential equation
\begin{equation*}
dl_g^{-1}(g') = u, \quad g(0) = e. 
\end{equation*}
By definition we have $\Phi(\hat{X}) = \exp X$ where $\hat{X}$ denotes the constant path with value $X \in \mathfrak{g}$. 
Consider the $(G \times G)$-action on $G$ defined by $(b, c) \cdot a := bac^{-1}$. Then
\begin{equation}\label{equiv1}
\Phi(g* u) = (g(0), g(1)) \cdot \Phi(u)
\end{equation}
where $g \in \mathcal{G}$ and $u \in V_\mathfrak{g}$. Moreover, for a Lie subgroup $U$ of $G \times G$, 
\begin{equation} \label{equiv2}
P(G, U) * u = \Phi^{-1}(U \cdot \Phi(u)).
\end{equation}
Furthermore $\Phi$ becomes a principal $P(G, \{e\} \times \{e\})$-bundle.

The differential of $\Phi : V_\mathfrak{g} \rightarrow G$ at $\hat{0} \in V_\mathfrak{g}$ is given by
\begin{equation}\label{diff6}
(d \Phi)_{\hat{0}} (X) = \int_0^1 X(t) dt
\end{equation}
where $X \in V_\mathfrak{g} \cong T_{\hat{0}} V_\mathfrak{g}$. Write $\hat{\mathfrak{g}}$ for the space of constant paths with values in $\mathfrak{g}$ and $F_u$ for the fiber of $\Phi$ through $u \in V_\mathfrak{g}$.  From \eqref{diff6} we have 
\begin{equation}\label{decomp97} 
T_{\hat{0}} V_\mathfrak{g} = \hat{\mathfrak{g}} \oplus T_{\hat{0}} (F_{\hat{0}})
, \qquad
X = \int_0^1X(t) dt \oplus \left(X - \int_0^1X(t) dt \right)
\end{equation}
From \eqref{equiv1} we know that $g*$ maps fibers of $\Phi$ to fibers of $\Phi$. Thus
\begin{equation*}
T_{g * \hat{0}} V_\mathfrak{g} = \operatorname{Ad}(g) \hat{\mathfrak{g}} \oplus T_{g * \hat{0}} (F_{g * \hat{0}})
\end{equation*}
for any $g \in \mathcal{G}$. Therefore the horizontal distribution $\mathcal{H}^\Phi$ of $\Phi$ is defined by 
\begin{equation*}
\mathcal{H}^{\Phi} (g * \hat{0}) := \operatorname{Ad}(g) \hat{\mathfrak{g}}.
\end{equation*}

We write $D$ for the flat connection on $V_\mathfrak{g}$ and $\nabla^G$ for the canonical connection on $G \cong (G \times G) /\Delta G$. The following fact was shown (\cite[Theorem 6.6]{M6}): 
\begin{theorem}[\cite{M6}] \label{previousemain}\normalsize
Let $G$ be a connected Lie group. Then the parallel transport map $\Phi : (V_\mathfrak{g}, D) \rightarrow (G, \nabla^G)$ is an affine submersion with horizontal distribution $\mathcal{H}^{\Phi}$. Moreover the fundamental tensor $\mathcal{A}^{\Phi}$ restricted to $\mathcal{H}^\Phi \times \mathcal{H}^\Phi$ is alternating.
\end{theorem}

Let $N = G/K$ be a homogeneous space. Then the \textit{parallel transport map $\Phi_{N}$ over $N$} is defined as the composition
\begin{equation*}
\Phi_{N} : = \pi \circ \Phi : V_\mathfrak{g} \rightarrow G  \rightarrow N.
\end{equation*}
Note that if $K = \{e\}$ then $\Phi_{N} = \Phi$. 
Note also that $\Phi_{(G \times G)/ \Delta G}$ is naturally identified with $\Phi$ (cf.\ \cite{M5}). By definition we have $\Phi_{N}(\hat{X}) = (\exp X)K$ for $X \in \mathfrak{g}$. Consider the $G$-action on $N$ defined by $b \cdot aK := (ba)K$. By \eqref{equiv1} we have
\begin{equation}\label{equiv3}
\Phi_{N} (g* u) = g(0)\cdot \Phi_{N}(u)
\end{equation}
where $g \in P(G, G \times K)$ and $u \in V_\mathfrak{g}$. Moreover, by \eqref{equiv2} we have 
\begin{equation}\label{equiv4}
P(G, H \times K) * u = \Phi_{N}^{-1}(H \cdot \Phi_{N}(u))
\end{equation}
where $H$ is a Lie subgroup of $G$. Furthermore the following fact holds. Here we do not suppose that $N$ is reductive, unlike in \cite[Proposition 7.3 (ii)]{M6}.

\begin{proposition}\label{principalbdl5}\normalsize
$\Phi_{N}$ becomes a principal $P(G, \{e\} \times K)$-bundle.
\end{proposition}
\begin{proof}\normalsize
The map $P(G, \{e\} \times G) \rightarrow V_\mathfrak{g}$, $g \mapsto dl_g^{-1}(g')$ is diffeomorphism \cite[Lemma 6.2]{M6}. Under this identification we have $\Phi_{N}(g) = g(1)K$. Moreover the $P(G, \{e\} \times G)$-action on $V_\mathfrak{g}$ is identified with the $P(G, \{e\} \times G)$-action on itself by $h \cdot g := gh^{-1}$. Take a subspace $\mathfrak{p}$ of $\mathfrak{g}$ satisfying $\mathfrak{g} = \mathfrak{k} \oplus \mathfrak{p}$. ($\mathfrak{p}$ need not be $\operatorname{Ad}(K)$-invariant.) By the inverse function theorem there exist open neighborhoods $W$ of $0 \in \mathfrak{p}$ and $U$ of $eK \in N$ such that the map $W \rightarrow U$, $X \mapsto (\exp X)K$ is a diffeomorphism. For each $g \in P(G, \{e\} \times \pi^{-1}(U))$ we define $X_g \in W$ by $g(1) K = (\exp X_g)K$. Then
\begin{equation*}
P(G, \{e\} \times \pi^{-1}(U)) \rightarrow U \times P(G, \{e\} \times K)
, \quad
g \mapsto (g(1)K, g(t)^{-1} \exp tX_g)
\end{equation*}
is an equivariant diffeomorphism. This proves the proposition.
\end{proof}

Let $N = G/K$ be a reductive homogeneous space with decomposition $\mathfrak{g} = \mathfrak{k} \oplus \mathfrak{p}$. By \eqref{diff6} the differential of $\Phi_{N}$ at $\hat{0} \in V_\mathfrak{g}$ is given by
\begin{equation}\label{diff7}
(d \Phi_{N})_{\hat{0}} (X) = \int_0^1 X(t)_\mathfrak{p} dt
\end{equation}
where $X \in V_\mathfrak{g} \cong T_{\hat{0}} V_\mathfrak{g}$. Write $\hat{\mathfrak{p}}$ for the space of constant paths with values in $\mathfrak{p}$ and $\mathcal{F}_{u}$ for the fiber of $\Phi_{N}$ through $u \in V_\mathfrak{g}$.  From \eqref{diff7} we have
\begin{equation*}
T_{\hat{0}} V_\mathfrak{g} = \hat{\mathfrak{p}} \oplus T_{\hat{0}} (\mathcal{F}_{\hat{0}}). 
\end{equation*}
From \eqref{equiv3} we know that $g * $ maps fibers of $\Phi_{N}$ to fibers of $\Phi_{N}$. Thus
\begin{equation*}
T_{g * \hat{0}} V_\mathfrak{g} = \operatorname{Ad}(g) \hat{\mathfrak{p}} \oplus T_{g * \hat{0}} (\mathcal{F}_{g * \hat{0}})
\end{equation*}
where $g \in P(G, G \times K)$. Since $\mathfrak{p}$ is invariant under $\operatorname{Ad}(K)$ the horizontal distribution $\mathcal{H}^{\Phi_{N}}$ of $\Phi_{N}$ is well-defined by 
\begin{equation}\label{hdist98}
\mathcal{H}^{\Phi_{N}} (g * \hat{0}) := \operatorname{Ad}(g) \hat{\mathfrak{p}}
\end{equation}
where $g \in P(G, G \times K)$. 

To prove Theorem \ref{mainthm1} we show the following lemma, which generalizes Proposition 7.5 of \cite{M6}. Here $l_a$ denotes the left translation by $a \in G$.
\begin{lemma}\label{keylem}\normalsize
Let $N = G/K$ be a reductive homogeneous space with decomposition $\mathfrak{g} = \mathfrak{k} \oplus \mathfrak{p}$ and with natural torsion-free connection $\nabla^{N}$. Write $\nabla^G$ for the canonical connection of $G \cong (G \times G)/ \Delta G$. Then $\pi : (G, \nabla^G) \rightarrow (N, \nabla^{N})$ is an affine submersion with horizontal distribution $\mathcal{H}^\pi(a) := dl_a(\mathfrak{p})$. Moreover the fundamental tensor $\mathcal{A}^\pi$ restricted to $\mathcal{H}^\pi \times \mathcal{H}^\pi$ is alternating. 
\end{lemma}

\begin{proof}\normalsize
It is clear that the following diagram commutes: 
\begin{equation}\label{commute3}
\begin{CD}
(G, \nabla^G) @>l_a >> (G, \nabla^G)
\\
@V\pi  VV  @V\pi  VV
\\
(N, \nabla^{N}) @>L_{a}>> (N, \nabla^{N}). 
\end{CD}
\end{equation}
Note that $l_a$ and $L_a$ are affine transformations and $\mathcal{H}^\pi$ is invariant under $l_a$. Thus, to show that $\pi : (G, \nabla^G) \rightarrow (N, \nabla^{N})$ is an affine submersion with horizontal distribution $\mathcal{H}^\pi$, we have only to show 
\begin{equation}\label{wantshow2}
(\nabla^G_{X_e} \bar{Z})^{\mathcal{H}}
=
(\nabla^{N}_{X^*_{eK}} Z \overline{)}
\end{equation}
for $X \in \mathfrak{p}$ and $Z \in \mathfrak{X}(N)$. Here $\bar{Z}$ denotes the horizontal lift of $Z$ and the superscript $\mathcal{H}$ in the left term denotes the projection onto $\mathcal{H}^\pi(e) = \mathfrak{p}$.

Fix the local trivialization given in Lemma \ref{lem1}. Take a basis $\{Y_i\}_{i = 1}^n$ of $\mathfrak{p}$ and 
write $Z|_U = \sum_{i = 1}^n \varphi_i Y_i^\#$ where $\varphi_i \in C^\infty(U)$. By \eqref{localexp} we have
\begin{equation}
\nabla^{N}_{X^\#_{eK}} (Z|_{U})
= 
\sum_{i =1}^n \left(\frac{1}{2} \varphi_i(eK) [X, Y_i]_\mathfrak{p} + (X^\#_{eK} \varphi_i) (Y_i^\#)_{eK} \right).
\end{equation}
Write $W_i$ for the horizontal lift of $Y_i^\#$. 
Note that $W_i$ and $Y_i$ are equal on $\tilde{U}$, but not necessarily equal on the whole $\pi^{-1}(U)$. 
We have $\bar{Z}|_{\pi^{-1}(U)} = \sum_{i = 1}^n (\varphi_i \circ \pi) W_i$.
By the Leibniz rule we have
\begin{equation*}
\nabla^G_{X_e} (\bar{Z}|_{\pi^{-1}(U)})
=
\sum_{i =1}^n \left(\varphi_i(eK) \nabla^G_{X_e} W_i + (X^{\#}_{eK} \varphi_i)(Y_i)_e \right).
\end{equation*}
Let $\iota : \tilde{U} \rightarrow G$ denote the inclusion map. Then
\begin{equation*}
\nabla^G_{X_e} W_i
=
\nabla^{\iota^*TG}_{X_e}  (\iota^* W_i)
=
\nabla^{\iota^*TG}_{X_e}  (\iota^* Y_i)
=
\nabla^{G}_{X_e} Y_i
=
\frac{1}{2}[X, Y_i]_e
\end{equation*}
by Corollary \ref{canoconng}. Here $\iota^*$ denotes the pullback. Thus \eqref{wantshow2} follows.

By Corollary \ref{canoconng} the fundamental tensor $\mathcal{A}^\pi$ is given by
$
\mathcal{A}^\pi(X, Y) = \frac{1}{2}[X, Y]_\mathfrak{k}
$
where $X, Y \in \mathfrak{p}$. Thus $\mathcal{A}^\pi$ restricted to $\mathcal{H}^\pi \times \mathcal{H}^\pi$ is alternating. 
\end{proof}

We are now in a position to prove Theorem \ref{mainthm1}. 
\begin{proof}[Proof of Theorem \textup{\ref{mainthm1}}]\normalsize
From \eqref{equiv3} the diagram
\begin{equation*}\label{commute90}
\begin{CD}
(V_\mathfrak{g}, D) @>g* >> (V_\mathfrak{g}, D)
\\
@V\Phi_{N} VV  @V\Phi_{N} VV
\\
(N, \nabla^{N}) @>L_{g(0)}>> (N, \nabla^{N})
\end{CD}
\end{equation*}
commutes for any $g \in P(G, G \times \{e\})$. Note that $g*$ and $L_{g(0)}$ are affine transformations and $\mathcal{H}^{\Phi_{N}}$ is invariant under $g*$. Thus, to show that $\Phi_{N}$ is an affine submersion with horizontal distribution $\mathcal{H}^{\Phi_{N}}$ we have only to show
\begin{equation}\label{wantoshow}
(D_{\hat{X}} \hat{Z})^{\mathcal{H}}
=
(\nabla^{N}_{X^*_{eK}} Z\widehat{)}
\end{equation}
for $X \in \mathfrak{p}$ and $Z \in \mathfrak{X}(N)$. Here $\hat{Z}$ denotes the horizontal lift of $Z$ and the superscript $\mathcal{H}$ in the left term denotes the projection onto $\mathcal{H}^{\Phi_{N}}(\hat{0}) = \hat{\mathfrak{p}}$.

From \eqref{equiv1} we have the commutative diagram for any $g \in P(G, G \times \{e\})$: 
\begin{equation*}\label{commute91}
\begin{CD}
(V_\mathfrak{g}, D) @>g* >> (V_\mathfrak{g}, D)
\\
@V\Phi VV  @V\Phi VV
\\
(G, \nabla^G) @>l_{g(0)}>> (G, \nabla^G).
\end{CD}
\end{equation*}
The horizontal distributions $\mathcal{H}^\Phi$ and $\mathcal{H}^\pi$ are invariant under $g*$ and $l_{g(0)}$ respectively. This implies 
$\hat{Z} = (\bar{Z} \widetilde{)}$
where the tilde denotes the horizontal lift with respect to $\Phi$. Thus, by Theorem \ref{previousemain} and Lemma \ref{keylem} we have 
\begin{equation*}
(D_{\hat{X}} \hat{Z})^{\mathcal{H}}
=
(\nabla^G_{X_e} \bar{Z} \widetilde{)_\mathfrak{p}}
=
(\nabla^{N}_{X^*_{eK}} Z\widehat{)}
\end{equation*}
which proves \eqref{wantoshow}. Since $T_{\hat{0}} (\mathcal{F}_{\hat{0}}) = T_{\hat{0}} (F_{\hat{0}}) \oplus \hat{\mathfrak{k}}$ and $\Phi$ is an affine submersion with horizontal distribution $\mathcal{H}^\Phi$ we have 
\begin{equation*}
\mathcal{A}^{\Phi_{N}}(\hat{X}, \hat{Y})
=
\mathcal{A}^\Phi(\hat{X}, \hat{Y}) \oplus \mathcal{A}^\pi(X, Y)
\end{equation*}
for $X, Y \in \mathfrak{p}$. Hence Theorem \ref{previousemain} and Lemma \ref{keylem} imply that $\mathcal{A}^{\Phi_{N}}$ restricted to $\mathcal{H}^{\Phi_{N}} \times \mathcal{H}^{\Phi_{N}}$ is alternating. This completes the proof.
\end{proof}

\section{Compactness of the shape operator}\label{compactshape}

In this section, we study the shape operators of fibers of the parallel transport map and prove Theorem \ref{mainthm2}. 

Let $G$ be a connected Lie group with Lie algebra $\mathfrak{g}$ and $\Phi : V_\mathfrak{g} \rightarrow G$ the parallel transport map. The Lie subgroup
$
P(G ,\{e\} \times \{e\})
$
of $\mathcal{G} = H^1([0,1], G)$ has Lie algebra
\begin{equation*}
\operatorname{Lie} P(G, \{e\} \times \{e\}) := \{Z \in H^1([0,1], \mathfrak{g}) \mid Z(0) = Z(1) = 0\}
\end{equation*}
and acts on the fiber $F = F_{\hat{0}} = \Phi^{-1}(e)$ simply transitively. The exponential map $\exp^\mathcal{G}$ of $\mathcal{G}$ is defined by 
$
(\exp^\mathcal{G} Z)(t) := \exp^G Z(t)
$
for
$
Z \in H^1([0,1], \mathfrak{g})
$
where $\exp^G$ denotes the exponential map of $G$. 
Since
\begin{equation}\label{deriv6}
\left. \frac{d}{ds} \right|_{s = 0} (\exp^\mathcal{G} sZ) * \hat{0} = -Z' 
\end{equation}
we have
\begin{equation}\label{fibertang}
T_{\hat{0}} F = \{ Z'  \in V_\mathfrak{g} \mid Z \in \operatorname{Lie} P(G, \{e\} \times \{e\})\}.
\end{equation}

For $\xi \in \mathfrak{g}$ we write $\mathcal{E}(\hat{\xi})$ for the $P(G, \{e\} \times \{e\})$-equivariant vector field on $F$: 
\begin{equation*}
\mathcal{E} (\hat{\xi}) (g* \hat{0})
:=
\operatorname{Ad}(g) \hat{\xi}
, \qquad 
g \in P(G, \{e\} \times \{e\}).
\end{equation*}
Then we have
\begin{equation*}
- D_{Z'} \mathcal{E}(\hat{\xi})
=
\left. \frac{d}{ds} \right|_{s = 0}
\mathcal{E} (\hat{\xi}) ((\exp^\mathcal{G} sZ )* \hat{0})
=
[Z, \hat{\xi}].
\end{equation*}
From this and \eqref{decomp97} we get an explicit formula for the shape operator $A^F_{\hat{\xi}}$ of $F$ (\cite[Lemma 1]{M1}): 
\begin{equation}\label{shapeop9}
A_{\hat{\xi}}^F (Z')
=
 [Z, \hat{\xi}] - \int_0^1 \left[Z(t), \xi \right]  dt
\end{equation}
where $Z \in \operatorname{Lie} P(G, \{e\} \times \{e\})$.

We will also use the following inequality: 
\begin{equation}\label{ineq}
\left|
\int_0^t X(s) ds
\right|^2
\leq
t \int_0^t |X(s)|^2 ds
\leq
t \int_0^1 |X(s)|^2 ds
\end{equation}
where $X \in L^2([0,1], \mathbb{R}^d)$ and $t \in [0,1]$. This is almost trivial by the Schwartz inequality: By decomposing $X = \sum_i X_i e_i$ by the canonical basis $\{e_i\}_i$ of $\mathbb{R}^d$ we have 
\begin{align*}
\left|
\int_0^t X(s) ds
\right|^2
&=
\sum_i
\left(
\int_0^t  X_i(s) ds 
\right)^2
\leq 
\sum_i
t \int_0^t  X_i(s)^2 ds 
=
t \int_0^t   |X(s)|^2 ds.
\end{align*}

 We are now in a position to prove Theorem \ref{mainthm2}.

\begin{proof}[Proof of Theorem \textup{\ref{mainthm2}}]\normalsize
Since the assertion is local, we may assume that $M$ is embedded in $N$ and identify $\hat{M}$ with $\Phi_{N}^{-1}(M)$. Since $\Phi_{N}^{-1}(M) = \Phi^{-1}(\pi^{-1}(M))$ it suffices to prove the assertion when $K = \{e\}$. 
By left translation we may assume $e \in M$ and $\hat{0} \in \hat{M}$.
Let $A^{\hat{M}}_{\hat{\xi}}$ denote the shape operator of $\hat{M}$ where $\xi \in \mathcal{W}_e$. 
For $X \in T_{\hat{0}} \hat{M}$ we write $X^\mathcal{H}$ and $X^\mathcal{V}$ for the horizontal and vertical components respectively. Then
\begin{equation}\label{shapedecomp}
A^{\hat{M}}_{\hat{\xi}} (X)
=
A^{\hat{M}}_{\hat{\xi}} (X^\mathcal{H})^\mathcal{H} + A^{\hat{M}}_{\hat{\xi}} (X^\mathcal{H})^\mathcal{V} 
+ 
A^{\hat{M}}_{\hat{\xi}} (X^\mathcal{V})^\mathcal{H} + A^{\hat{M}}_{\hat{\xi}} (X^\mathcal{V})^\mathcal{V}.
\end{equation}
In the right side, the last term is identified with $A^F_{\hat{\xi}}(X^\mathcal{V})$ and the other terms span the finite dimensional subspace. Thus we have only to show that $A^F_{\hat{\xi}}$ is compact. 

Choose an inner product $\langle \cdot, \cdot \rangle$ on $\mathfrak{g}$. Write $\langle \cdot, \cdot \rangle_{L^2}$ for the $L^2$-inner product on $V_\mathfrak{g}$. Let $\{Z_n\}_n \subset \operatorname{Lie} P(G, \{e\} \times \{e\})$ be a sequence such that $\|Z'_n\|_{L^2}$ is bounded. Since $Z_n(0) = 0$ we have $Z_n(t) = \int_0^t Z_n'(s) ds$. By using \eqref{ineq} we have
\begin{equation*}
\|Z_n\|_{L^2}^2
=
\int_0^1 \left|\int_0^t Z'_n(s) ds \right|^2 dt
\leq 
\int_0^1 t  \,dt \int_0^1 |Z'_n(s)|^2 ds 
=
\frac{1}{2}
\|Z'_n\|_{L^2}^2.
\end{equation*}
Thus
$
\|Z_n\|_{H^1} 
$
is bounded. By the Sobolev embedding theorem there exists a subsequence $\{Z_{n(k)}\}_k$ which converges uniformly to some $Z_\infty \in C([0,1], \mathfrak{g})$. Clearly $Z_{n(k)}$ converges to $Z_\infty$ also in the $L^2$-sense. Thus \eqref{shapeop9} shows that $A_{\hat{\xi}}^F(Z'_{n(k)})$ converges to $ [Z_\infty, \hat{\xi}] - \int_0^1[Z_\infty(t), \xi] dt \in \operatorname{Ker} (d \Phi)_{\hat{0}}  =T_{\hat{0}} F $. This completes the proof. 
\end{proof}

\begin{remark}\label{remimpor}\normalsize
Let $\nabla^{N}$ be the natural torsion-free connection on $N = G/K$. By Theorem \ref{mainthm1} the parallel transport map $\Phi_{N}: (V_\mathfrak{g}, D)  \rightarrow (N, \nabla^{N})$ becomes an affine submersion with horizontal distribution. Thus, in this case, the first term in the right side of \eqref{shapedecomp} is identified with $A_\xi^M((d \Phi)_{\hat{0}}(X))$ where $A_\xi^M$ denotes the shape operator of the submanifold $M$ of $(N, \nabla^{N})$ with transversal bundle $\mathcal{W}$ (cf.\ \cite[(3.4)]{M6}). 
\end{remark}

Motivated by Theorem \ref{mainthm2} we make the following definition. Note that this terminology is different from that in the author's another paper \cite[Section 4]{M6}.

\begin{definition}\normalsize
Let $V$ be a Hilbertable space equipped with the flat connection $D$ and $M$ an immersed submanifold of $V$ with transversal bundle $\mathcal{W}$. Then $M$ is called \textit{Fredholm} if $M$ has finite codimension in $V$ and satisfies (B) (and thus (A)) in the introduction.
\end{definition}

Let $M$ be a Fredholm submanifold of a Hilbertable space $V$ with transversal bundle $\mathcal{W}$. For each $(p, \xi) \in \mathcal{W}$ the shape operator $A_\xi : T_p M \rightarrow T_pM$ is a compact operator, which is not self-adjoint in general. Thus we consider the complexification $A_\xi^\mathbb{C}$.
From the spectral theory of compact operators on Banach spaces (cf.\ \cite{Con90, Rob20}) we have: 
\begin{lemma}\normalsize
The following properties hold:
\begin{enumerate}
\item the nonzero spectrum of $A_\xi^{\mathbb{C}}$ is finite or countably infinite and consists of eigenvalues $\{\lambda_k\}_k$ of $A_\xi^{\mathbb{C}}$,
\item each nonzero eigenvalue $\lambda_k$ of $A_\xi^{\mathbb{C}}$ has finite multiplicity,
\item if $V$ has infinite dimension, then $0$ belongs to the spectrum of $A_\xi^{\mathbb{C}}$,
\item if $\{\lambda_k\}_k$ is countably infinite, then $\lim_{k \rightarrow \infty} \lambda_k = 0$.
\end{enumerate}
\end{lemma}
Note that the shape operator $A_\xi$ is not of trace class in general, even when $M$ is a (Riemannian) PF submanifold of  a Hilbert space. In the remaining two sections, we will study the trace of the shape operator (i.e.\ the mean curvature) of Fredholm submanifolds in Hilbertable spaces.

\section{Regularized mean curvatures I}\label{minimality}

In this section, we propose a definition of the mean curvatures of Fredholm submanifolds in Hilbertable spaces and discuss their properties. For the Riemannian case, we refer to King-Terng \cite{KT93}, Heintze-Liu-Olmos \cite{HLO06} and Koike \cite{Koi02}.

Let $H$ be a (real, separable) Hilbertable space and $A: H \rightarrow H$ a compact operator. Write $A^\mathbb{C} : H^\mathbb{C} \rightarrow H^\mathbb{C}$ for the complexification and $\Lambda$ for the set of all eigenvalues of $A^\mathbb{C}$. Observe that if $\lambda \in \Lambda$ then the complex conjugation $\bar{\lambda}$ belongs to $\Lambda$. Thus, to define the trace of $A$, it is reasonable to focus on the \textit{real} part of each eigenvalue. Write $\operatorname{Re} \lambda$ and $\operatorname{Im} \lambda$ for the real and imaginary parts of $\lambda$ respectively. Set $\Lambda_{>0} := \{\operatorname{Re} \lambda \mid \lambda \in \Lambda, \ \operatorname{Re} \lambda > 0\}$ and $\Lambda_{<0} := \{\operatorname{Re} \lambda \mid \lambda \in \Lambda, \ \operatorname{Re} \lambda < 0\}$. Write $\Lambda_{>0} = \{\mu_k\}_{k = 1}^\infty$ and $\Lambda_{<0} = \{\nu_k\}_{k = 1}^\infty$ (counted with multiplicities) so that
\begin{equation*}
\nu_1 \leq \nu_2 \leq \cdots < 0 < \cdots \leq \mu_2 \leq \mu_1
\end{equation*}
where we regard $\mu_{k} = 0$ (resp.\ $\nu_k = 0$) when the set $\Lambda_{>0}$ (resp.\ $\Lambda_{<0}$) is finite and has cardinality less than $k$.

\begin{definition}\normalsize
A compact operator $A : H \rightarrow H$ is called \textit{regularizable} if the series
\begin{equation*}\label{hlotrace}
\sum_{k = 1}^\infty (\mu_k + \nu_k)
\end{equation*}
converges. Then its value is called the \textit{regularized trace} of $A$ and written as $\operatorname{tr}_r A$. 
We say that $A$ is \textit{regularizable} if its regularized trace exists.
\end{definition}

\begin{remark}\normalsize
When $A$ is a compact self-adjoint operator on a (real, separable) Hilbert space the above definition coincides with the regularized trace considered by Heintze, Liu and Olmos \cite{HLO06}, except that we do not assume $\operatorname{tr} (A^2) < \infty$ here.
\end{remark}

\begin{remark}\normalsize
Let $H$ be a Hilbertable space and $A : H \rightarrow H$ be a compact operator. Write $\{\lambda_k\}_k$ for the distinct eigenvalues of $A^\mathbb{C}$ so that for each $k$ one of the following conditions hold:
(i) $|\lambda_k| > |\lambda_{k + 1}|$; 
(ii) $|\lambda_k| = |\lambda_{k + 1}|$ and $\operatorname{Re} \lambda_k  > \operatorname{Re} \lambda_{k + 1}$;
(iii) $|\lambda_k| = |\lambda_{k + 1}|$, $\operatorname{Re} \lambda_k  = \operatorname{Re} \lambda_{k + 1}$ and $\operatorname{Im}  \lambda_k = - \operatorname{Im} \lambda_{k+1} > 0$.
Write $m(\lambda_k)$ for the multiplicity of $\lambda_k$. 
Koike \cite{Koi02, Koi04} defines the trace of $A^\mathbb{C}$ as
$
\sum_{k= 1}^\infty m(\lambda_k) \lambda_k
$
when it converges. (More precisely, he assumes that $A$ is the shape operator of a Fredholm submanifold and $A^\mathbb{C}$ has the basis of eigenfunctions.) This definition is different from ours. For example, if $A$ has eigenvalues $\{- \frac{1}{k}\}_{k = 1}^\infty \cup \{0\} \cup \{\frac{1}{k}\}_{k = 1}^\infty$ with $m(\frac{1}{k}) = m(- \frac{1}{k}) = k$, then the trace is $0$ in our sense, but does not converge in his sense.
\end{remark}

Let $M$ be a Fredholm submanifold of a Hilbertable space $V$ with transversal bundle $\mathcal{W}$. Then $M$ is called \textit{regularizable} if for each $(p, \xi) \in \mathcal{W}$ the shape operator $A_\xi$ is regularizable. Then $\operatorname{tr}_r A_\xi$ is called the \textit{regularized mean curvature} of $M$ in the direction of $\xi$. If $M$ is regularizable and $\operatorname{tr}_r A_\xi = 0$ for any $(p, \xi) \in \mathcal{W}$, then $M$ is called \textit{minimal}. 

\begin{proposition}\label{minimalfiber}\normalsize
Let $N = G/K$ be a reductive homogeneous space with decomposition $\mathfrak{g} = \mathfrak{k} \oplus \mathfrak{p}$. Then each fiber of the parallel transport map $\Phi_{N} : V_\mathfrak{g} \rightarrow N$ is minimal. 
\end{proposition}
\begin{proof}\normalsize
It is shown in \cite[Theorem 8.3]{M6} that each fiber $\mathcal{F}$ of $\Phi_N$ is a weakly reflective submanifold of $V_\mathfrak{g}$. Thus $\mathcal{F}$ is austere, that is, for each normal vector $\hat{\xi}$ the set of eigenvalues with multiplicities of the shape operator  $A^{\mathcal{F}}_{\hat{\xi}}$ is invariant under the multiplication by $(-1)$ \cite[Proposition 8.2]{M6}. Thus $\operatorname{tr}_r A^{\mathcal{F}}_{\hat{\xi}}$ exists and equal to $0$, which imply that $\mathcal{F}$ is regularizable and minimal.
\end{proof}
In the above proof, we have strongly used \cite[Theorem 8.3]{M6}. On the other hand, Theorem \ref{mainthm2} in the present paper is also essential to define ${\mathcal{F}}$ to be minimal.

In the Riemannian case, the minimality of fibers plays important roles in the lifting of minimal submanifolds and isoparametric submanifolds to a Hilbert space \cite{KT93, HLO06}. Then it is natural to ask whether analogous results hold in the affine case. 
To study this, it is fundamental to consider the following question, which makes sense by Theorem \ref{mainthm2}.

\begin{question}\label{que1}\normalsize
Let the notation be as in Theorem \ref{mainthm2}. Equip $N$ with the natural torsion-free connection. Take $(\hat{p}, \hat{\xi}) \in \hat{\mathcal{W}}$ and set $\xi: =d\Phi_N(\hat{\xi})$. Does the regularized trace of $A^{\hat{M}}_{\hat{\xi}}$ exist and coincide with the trace of $A_\xi^M$?
\end{question}

Note that Lemma 5.2 of \cite{HLO06} gives an affirmative answer to Question \ref{que1} when $N$ is a compact normal homogeneous space (see also \cite[Theorem 4.12]{KT93}).  

By Theorem \ref{mainthm1} we can reduce Question \ref{que1} to an analytic problem: 

\begin{question}\label{que2}\normalsize
Let $H$ be a Hilbertable space, $A: H \rightarrow H$ a regularizable compact operator and $B: H \rightarrow H$ a finite-rank operator.  
Is $A + B$ regularizable? 
Does 
$
\operatorname{tr}_r (A + B) = \operatorname{tr}_r A + \operatorname{tr} B
$ hold?
\end{question}

In fact, by Theorem \ref{mainthm1} the first term in the right hand side of \eqref{shapedecomp} is identified with $A^M_\xi ((d\Phi)_{\hat{0}} (X))$ (Remark \ref{remimpor}). This together with Proposition \ref{minimalfiber} shows that if the answer of Question \ref{que2} is yes, then so is the answer of Question \ref{que1}. Note that Lemma 4.5 of \cite{HLO06} gives an affirmative answer to Question \ref{que2} when $A$ and $B$ are \textit{self-adjoint} operators on a Hilbert space (see also \cite[Theorem 4.2]{KT93}).

\section{Regularized mean curvatures II}\label{finalsec}

In this section, we propose another definition of the mean curvatures of Fredholm submanifolds in Hilbertable spaces and study their properties.  This definition is valid for Fredholm submanifolds obtained through the parallel transport map.

In the finite dimensional case, the trace of a linear map is defined as the sum of diagonal entries of the matrix representation with respect to chosen bases and is shown to be independent of the bases. We want to do the same for compact operators on an infinite dimensional Hilbertable space. However, since the operators are not of trace class in general, we cannot expect such an independent definition. Thus we will take a basis which is compatible with the structure of the parallel transport map, define the trace as the series of  diagonal entries and show some independent property. 

To begin with, we make the following definition:
\begin{definition}\normalsize
Let $H$ be an infinite dimensional (real, separable) Hilbertable space and $A : H \rightarrow H$ a compact operator. 
Let $\mathcal{B} := \{e_j\}_{j = 1}^\infty$ be a (Schauder) basis of $H$. 
For each $j$, let $\{a_{ij} \in \mathbb{R}\}_{i = 1}^\infty$ denote the sequence satisfying $A e_j = \sum_{i = 1}^\infty a_{ij} e_i$. 
We call $\{a_{ii}\}_{i = 1}^\infty$ the \textit{diagonal entries of $A$ with respect to $\mathcal{B}$}.
If the sum
$
\sum_{i = 1}^\infty a_{ii}
$
converges, then we say that $A$ is \textit{regularizable with respect to $\mathcal{B}$}. 
Then its value is called the \textit{regularized trace of $A$ with respect to $\mathcal{B}$} and written as $\operatorname{tr}^\mathcal{B}_r A$. 
\end{definition}

Suppose that an infinite dimensional Hilbertable space $H$ is decomposed into finite dimensional subspaces:
\begin{equation*}
H = \bigoplus_{i = 1}^\infty H_i,
\end{equation*}
that is, any $v \in H$ can be written uniquely in the form $v = \sum_{i = 1}^\infty v_i$ where $v_i \in H_i$. A basis $\mathcal{B}$ of $H$ is said to be \textit{compatible with the decomposition} if $\mathcal{B}_i := H_i \cap \mathcal{B}$ is a basis of $H_i$ for each $i$. Then we have $\mathcal{B} = \bigcup_{i = 1}^\infty \mathcal{B}_i$.

Let $N = G/K$ be a reductive homogeneous space with decomposition $\mathfrak{g} = \mathfrak{k} \oplus \mathfrak{p}$ and  $\Phi_N : V_\mathfrak{g} \rightarrow N$ the parallel transport map. We fix the basis
$
\{1, \sin 2 n \pi t, \cos 2 n \pi t\}_{n = 1}^\infty
$
of $L^2([0,1], \mathbb{R})$. In other words, we fix the decomposition
\begin{equation*}
L^2([0,1], \mathbb{R})
=
\mathbb{R} \oplus \bigoplus_{n = 1}^\infty (\mathbb{R} \sin 2 n \pi t \oplus \mathbb{R} \cos 2 n \pi t).
\end{equation*}
Then we have the induced decomposition
\begin{align*}
V_\mathfrak{g}
& =
\mathfrak{g} \otimes L^2([0,1], \mathbb{R})
=
\mathfrak{g} \oplus \bigoplus_{n = 1}^\infty (\mathfrak{g} \sin 2 n \pi t \oplus \mathfrak{g} \cos 2 n \pi t)
\end{align*}
where we abbreviate $\mathfrak{g} \otimes \mathbb{R} f $ as $\mathfrak{g} f$ for $f \in L^2([0,1], \mathbb{R})$. Then by \eqref{diff7} we obtain the induced decomposition of the vertical subspace $\mathcal{V}_{\hat{0}} = \operatorname{Ker} (d \Phi_N)_{\hat{0}}$ at $\hat{0} \in V_\mathfrak{g}$:
\begin{equation}\label{decomp1234}
\mathcal{V}_{\hat{0}}
=
\mathfrak{k} 
\oplus 
\bigoplus_{n = 1}^\infty 
(
\mathfrak{g} \sin 2 n \pi t \oplus \mathfrak{g}\cos 2 n \pi t
).
\end{equation}
To extend this decomposition to the vertical subspace $\mathcal{V}_u$ at any $u \in V_\mathfrak{g}$ it is natural to consider the $P(G, G \times K)$-action, since $\Phi_{G/K}$ is an equivariant submersion \eqref{equiv3}. More precisely, we take $g \in P(G, G \times K)$ satisfying $u = g * \hat{0}$ and decompose
\begin{align}\label{decomp789}
&
\mathcal{V}_{u}
=
\operatorname{Ad}(g)\mathfrak{k}  \oplus \bigoplus_{n = 1}^\infty 
(
\operatorname{Ad}(g)(\mathfrak{g} \sin 2 n \pi t) \oplus \operatorname{Ad}(g)(\mathfrak{g}\cos 2 n \pi t)
).
\end{align}
This decomposition does not depend on the choice of $g$.

Let $N$ be as above and $\varphi : M \rightarrow N$ an immersed submanifold with transversal bundle $\mathcal{W}$. By Theorem \ref{mainthm2} the lifted submanifold $\hat{\varphi} : \hat{M} \rightarrow V_\mathfrak{g}$ with lifted transversal bundle $\hat{\mathcal{W}}$ is Fredholm. 
For $\hat{p} = (p, u) \in \hat{M}$ we have the isomorphism $T_{\hat{p}} \hat{M} \cong (T_p M\hat{)} \oplus \mathcal{V}_u$. This together with \eqref{decomp789} shows:
\begin{align} \label{decomp90}
T_{\hat{p}} \hat{M}
\cong
(T_p M\hat{)} \oplus 
\operatorname{Ad}(g)\mathfrak{k}  \oplus \bigoplus_{n = 1}^\infty 
(
\operatorname{Ad}(g)(\mathfrak{g} \sin 2 n \pi t) \oplus \operatorname{Ad}(g)(\mathfrak{g}\cos 2 n \pi t)
).
\end{align}
\begin{theorem}\label{mainthm3}\normalsize
Let $N = G/K$ be a reductive homogeneous space with decomposition $\mathfrak{g} = \mathfrak{k} \oplus \mathfrak{p}$ and with natural torsion-free connection $\nabla^N$ and $\varphi : M \rightarrow N$ be an immersed submanifold with transversal bundle $\mathcal{W}$. Take $(\hat{p}, \hat{\xi}) \in \hat{\mathcal{W}}$. 
Let $\mathcal{B}_{\hat{p}}$ be any basis of $T_{\hat{p}} \hat{M}$ compatible with the decomposition \eqref{decomp90}. Then the regularized trace of the shape operator $A^{\hat{M}}_{\hat{\xi}}$ with respect to $\mathcal{B}_{\hat{p}}$ exists, does not depend on the choice of $\mathcal{B}_{\hat{p}}$ and coincides with the trace of the shape operator $A^M_\xi$ where $\xi := d \Phi_N(\hat{\xi})$.
\end{theorem}

\begin{remark}\label{main3remark}\normalsize
Theorem \ref{mainthm3} can be thought of an extension of the results of King-Terng \cite[Theorem 4.12]{KT93} when $N$ is compact, of Heintze-Liu-Olmos \cite[Lemma 5.2]{HLO06} in the case of Riemannian submersions and of Koike (\cite[Theorem 4.1]{Koi02}, \cite[Theorem C]{Koi04}) when $M$ is a curvature-adapted submanifold of a Riemannian symmetric space $N$ of compact or non-compact type. On the other hand, our definition of the regularized trace of the shape operator is valid only for the lifted submanifold $\hat{M}$ of $V_\mathfrak{g}$. 
\end{remark}

We first consider the case of the fiber $\mathcal{F} := \mathcal{F}_{\hat{0}} = \Phi_{N}^{-1}(eK)$. By Proposition \ref{principalbdl5} and \eqref{deriv6} we have
\begin{equation*}
T_{\hat{0}} \mathcal{F}
=
\{Z' \mid Z \in \operatorname{Lie} P(G, \{e\} \times K) \}
\end{equation*}
where $\operatorname{Lie} P(G, \{e\} \times K) := \{Z \in H^1([0,1], \mathfrak{g}) \mid  Z(0)= 0, \ Z(1) \in \mathfrak{k}\}$. By calculations similar to those in Section \ref{compactshape} we have 
\begin{equation}\label{shapeop11}
A_{\hat{\xi}}^{\mathcal{F}} (Z')
=
 [Z, \hat{\xi}] - \int_0^1 \left[Z(t), \xi \right]_{\mathfrak{p}} dt
\end{equation}
where $\xi \in \mathfrak{p}$ and $Z \in \operatorname{Lie} P(G, \{e\} \times K)$.
\begin{lemma}\label{lemtrace7}
Let $x \in \mathfrak{k}$ and $y \in \mathfrak{g}$. Then 
\begin{enumerate}
\item $A^{\mathcal{F}}_{\hat{\xi}}(x)= \left(t - \frac{1}{2}\right) [x, \xi]$.

\item $A^{\mathcal{F}}_{\hat{\xi}}(y \sin 2 n \pi t) = \frac{1}{2n \pi} ([y, \xi]_\mathfrak{k}  - [y, \xi] \cos 2 n \pi t)$.

\item $A^{\mathcal{F}}_{\hat{\xi}}(y \cos 2 n \pi t) = \frac{1}{2n \pi} [y, \xi] \sin 2 n \pi t$.
\end{enumerate}
\end{lemma}

\begin{proof}\normalsize
(i): Set $Z = t x$. Then $Z \in \operatorname{Lie} P(G, \{e\} \times K)$ and $Z' = x$. Thus
\begin{align*}
\textstyle
A^{\mathcal{F}}_{\hat{\xi}}(x)
&
\textstyle
= 
[ t x, \hat{\xi}] - \int_0^1 t [x, \xi]_{\mathfrak{p}} dt
=
t [x, \xi] - \frac{1}{2} [x, \xi]_{\mathfrak{p}}
=
\left(t - \frac{1}{2}\right) [x, \xi] 
\end{align*}
since $[\mathfrak{k}, \mathfrak{p}] \subset \mathfrak{p}$.
(ii): Set $Z = \frac{1}{2n \pi} y (1 - \cos 2 n \pi t)$. Then $Z \in \operatorname{Lie} P(G, \{e\} \times K)$ and $Z' = y \sin 2 n \pi t$. Thus
\begin{align*}
\textstyle
A^{\mathcal{F}}_{\hat{\xi}}(y \sin 2 n \pi t) 
&
\textstyle
= 
\frac{1}{2n \pi} [y, \xi](1 - \cos 2 n \pi t)  - \frac{1}{2n \pi} [y, \xi]_\mathfrak{p} \int_0^1(1 - \cos 2 n \pi t)  dt 
\\
&
\textstyle
=
\frac{1}{2n \pi} ([y, \xi]_\mathfrak{k}  - [y, \xi] \cos 2 n \pi t).
\end{align*}
(iii): Set $Z = \frac{1}{2n \pi} y \sin 2 n \pi t$. Then $Z \in \operatorname{Lie} P(G, \{e\} \times K)$ and $Z' = y \cos 2n \pi t$. Thus
\begin{align*}
\textstyle
A^{\mathcal{F}}_{\hat{\xi}}(y \cos 2n \pi t) 
&
\textstyle
= 
\frac{1}{2n \pi} [y, \xi] \sin 2 n \pi t  - \frac{1}{2n \pi} [y, \xi]_\mathfrak{p} \int_0^1 \sin 2 n \pi t  dt 
\\
&
\textstyle
=
\frac{1}{2n \pi} [y, \xi] \sin 2 n \pi t.
\end{align*}
This proves the lemma.
\end{proof}

\begin{lemma}\label{corfinal}\normalsize
Let $\mathcal{B}^\mathcal{F}$ be any basis of $\mathcal{V}_{\hat{0}} = T_{\hat{0}} \mathcal{F}$ compatible with the decomposition \eqref{decomp1234}. Then any diagonal entry of $A^{\mathcal{F}}_{\hat{\xi}} : T_{\hat{0}} \mathcal{F} \rightarrow T_{\hat{0}} \mathcal{F}$ with respect to $\mathcal{B}^\mathcal{F}$ is $0$.
\end{lemma}
\begin{proof}\normalsize
Lemma \ref{lemtrace7} together with $\int_0^1 (t-\frac{1}{2}) dt = 0$ implies $A^\mathcal{F}_{\hat{\xi}}(\mathfrak{k}) \subset \operatorname{Ker} (d \Phi)_{\hat{0}} = T_{\hat{0}} F$. 
Moreover, by Lemma \ref{lemtrace7} we have $A^\mathcal{F}_{\hat{\xi}}(\mathfrak{g} \sin 2 n \pi t) \subset \mathfrak{k} \oplus (\mathfrak{g} \cos 2 n \pi t)$
and $A^\mathcal{F}_{\hat{\xi}}(\mathfrak{g} \cos 2 n \pi t) \subset \mathfrak{g} \sin 2 n \pi t$. This proves the lemma.
\end{proof}

We are now in a position to prove Theorem \textup{\ref{mainthm3}}.

\begin{proof}[Proof of Theorem \textup{\ref{mainthm3}}]\normalsize
Since the assertion is local, we may assume that $M$ is embedded in $N$.
Since the decomposition \eqref{decomp789} and the horizontal distribution \eqref{hdist98} are invariant under the $P(G, G \times K)$-action, we may assume $u = \hat{0}$ without loss of generality. Consider the decomposition \eqref{shapedecomp} which we simply rewrite as
\begin{equation*}
(A :=) \ 
A^{\hat{M}}_{\hat{\xi}}
\ \cong \ 
A_{11} + A_{21} + A_{12} + A_{22}.
\end{equation*}
By Theorem \ref{mainthm1}, $A_{11}$ is identified with $A^M_\xi$ (Remark \ref{remimpor}). 
By definition, $A_{22}$ is identified with $A^{\mathcal{F}}_{\hat{\xi}}$. 
We can ignore $A_{12}$ and $A_{21}$ since they are linear maps between horizontal and vertical subspaces. 
Thus, by Lemma \ref{corfinal} the regularized trace of $A$ with respect to $\mathcal{B}_{\hat{p}}$ exists, equal to $\operatorname{tr} A^M_\xi$ and does not depend on the choice of  $\mathcal{B}_{\hat{p}}$.  This completes the proof.
\end{proof}

\begin{definition}\normalsize
Let the notation be as in Theorem \ref{mainthm3}. Then the \textit{regularized mean curvature} of the Fredholm submanifold $\hat{M}$ of $V_\mathfrak{g}$ in the direction of $\hat{\xi}$ is defined as the regularized trace of $A^{\hat{M}}_{\hat{\xi}}$ with respect to $\mathcal{B}_{\hat{p}}$. We say that $\hat{M}$ is \textit{minimal} if for each $(\hat{p}, \hat{\xi}) \in \hat{\mathcal{W}}$ the regularized mean curvature of $\hat{M}$ in the direction of  $\hat{\xi}$ vanishes.
\end{definition}

From Theorem \ref{mainthm3} we have:

\begin{corollary}\label{minimalfiber89}\normalsize
The Fredholm submanifold $\hat{M}$ of $V_\mathfrak{g}$ is minimal if and only if the submanifold $M$ of the reductive homogeneous space $N = G/K$ is minimal. In particular, each fiber of the parallel transport map $\Phi_N : V_\mathfrak{g} \rightarrow N$ is minimal.
\end{corollary}

\bigskip

\subsection*{Acknowledgements}
The author would like to thank Professors Yoshihiro Ohnita, Hiroshi Tamaru and Takashi Sakai for their interests in his work and constant supports. 
The author is also grateful to Professors E. Heintze, J.-H. Eschenburg and P. Quast for valuable comments and their hospitality during his visit to the University of Augsburg in July 2025.

\end{document}